\documentclass{amsart}
\usepackage{graphicx}
\usepackage{amssymb,latexsym,wasysym}
\usepackage{amsfonts,amsmath,amsthm}
\usepackage{color}
\usepackage{enumerate}
\usepackage[dvipsnames]{xcolor}
\usepackage{url}
\usepackage[colorlinks=true,
            linkcolor=violet,
            urlcolor=Aquamarine,
            citecolor=YellowOrange]{hyperref}
\newtheorem{thm}{Theorem}[section]

\newtheorem{lem}[thm]{Lemma}

\newtheorem{rem}[thm]{Remark}

\newcommand{\norm}[1]{\left\Vert#1\right\Vert}
\newcommand{\abs}[1]{\left\vert#1\right\vert}
\newcommand{\set}[1]{\left\{#1\right\}}
\newcommand{\Real}{\mathbb R}

\newcommand{\bx}{ {\bf x}}
\newcommand{\pfrac}[2]{\frac{\partial #1}{\partial #2}}

\numberwithin{equation}{section}

\begin{document}
\title{Direct minimizing method for Yang-Mills energy over $SO(3)$ bundle}

\author{Hao Yin}

\address{Hao Yin,  School of Mathematical Sciences,
University of Science and Technology of China, Hefei, China}
\email{haoyin@ustc.edu.cn }
\thanks{The research of Hao Yin is supported by NSFC 11971451.}
\maketitle
\begin{abstract}
In this paper, we use the direct minimizing method to find Yang-Mills connections for $SO(3)$ bundles over closed four manifolds. By constructing test connections, we prove that a minimizing sequence converges strongly to a minimizer under certain assumptions. In case the strong convergence fails, we find an anti-selfdual (or selfdual) connection.	
\end{abstract}

\section{Introduction}

Suppose that $M$ is a closed oriented $4$-manifold with arbitrary Riemannian metric $g$ and $G$ is a compact simple Lie group. For each  principal $G$-bundle $P$ over $M$, one defines two cohomology classes
\[
	\eta(P)\in H^2(M,\pi_1(G))\quad \text{and}\quad p_1(P)\in H^4(M,\mathbb Z)
\]
such that two bundles are isomorphic if and only if the two cohomology classes are the same. The first one $\eta(P)$ determines the restriction of the bundle to the $3$-skeleton of $M$ and the second one $p_1(P)$ is the first Pontryagin class, which is related to the curvature integral of a $G$-connection by the Chern-Weil theory.

For a given $G$-connection $D$ over $P$, the Yang-Mills functional is defined by
\[
	\mathcal Y \mathcal M(D)=\int_M \abs{F_D}^2 dV_g,
\]
where $F_D$ is the curvature form of $D$.
A critical point of $\mathcal Y \mathcal M$ is called a Yang-Mills connection. With the help of the Hodge star operator, any two form $\omega$ over $M$ is uniquely decomposed into the sum of selfdual and anti-selfdual forms
\[
\omega= \omega^+ + \omega^-.
\]
A connection $D$ is said to be anti-selfdual(ASD) if and only if
\[
(F_D)^+=0.
\]
Similarly, $D$ is selfdual(SD) if and only if $(F_D)^-=0$.
By the Chern-Weil theory, we have
\begin{equation}
p_1(P)= \frac{1}{4\pi^2} \left( \norm{F_D^+}^2 - \norm{F_D^-}^2 \right),
	\label{eqn:pontryagin}
\end{equation}
which implies that the selfdual (or anti-selfdual) connections realize the absolute minimum of the Yang-Mills functional. 

Anti-selfdual $SU(2)$ connections played an important role in the geometry of four manifolds since the groundbreaking paper \cite{donaldson1983application}. In \cite{fintushel1984}, Fintushel and Stern simplified the argument of Donaldson by using $SO(3)$ connections. Known examples of ASD/SD connections come from (1) explicit construction (see \cite{atiyah1978}), (2) the gluing method due to Taubes \cite{taubes1982self,taubes1984self} and (3) the complex category (see \cite{donaldson1990}). Variational methods \cite{sedlacek1982direct,sibner1989} were also applied to prove the existence of Yang-Mills connections. 

A very natural way of constructing Yang-Mills connection and ASD/SD connection is by direct minimizing. Exploiting the compactness theorem due to Uhlenbeck \cite{uhlenbeck1982connections}, Sedlacek \cite{sedlacek1982direct} studied the direct minimizing approach for the Yang-Mills functional over a four manifold. For our purpose, it is easier to work with vector bundles with structure group $G$ instead of principal bundles. For any $\eta\in H^2(M,\pi_1(G))$ and $p_1\in H^4(M,\mathbb Z)$, we denote by $E_{\eta,p_1}$ a vector bundle with structure group $G$, the principal bundle $P_E$ associated to which satisfies $\eta(P_E)=\eta$ and  $p_1(P_E)=p_1$. Let $\mathcal C_{\eta,p_1}$ be the set of all smooth $G$-connections of $E_{\eta,p_1}$ and 
\[
	\mathcal C_\eta= \bigcup_{p_1\in H^4(M,\mathbb Z)} \mathcal C_{\eta,p_1}.
\]
\begin{rem}
	The map from the isomorphism classes of principal $G$-bundles to $H^2(M,\pi_1(G))\times H^4(M,\mathbb Z)$ is not surjective in general. However, $\eta(P)$ can take all values in $H^2(M,\pi_1(G))$. When the pair $(\eta,p_1)$ is not in the image, we simply take $\mathcal C_{\eta,p_1}=\emptyset$.
\end{rem}
The main result in \cite{sedlacek1982direct} is
\begin{thm}\label{thm:1982}
	[Sedlacek, 1982] For each $\eta\in H^2(M,\pi_1(G))$, there is $p_1\in H^4(M,\mathbb Z)$ and a smooth Yang-Mills connection $D$ in $\mathcal C_{\eta,p_1}$ such that
	\[
		\mathcal Y \mathcal M(D)= \inf_{D'\in \mathcal C_{\eta}} \mathcal Y \mathcal M(D').
	\]
\end{thm}

Due to the possible concentration of Yang-Mills energy, even if one takes a minimizing sequence in $\mathcal C_{\eta,p_1}$ for some fixed $p_1\in H^4(M,\mathbb Z)$, the limit given by the Uhlenbeck compactness theorem may not be a connection on the bundle $E_{\eta,p_1}$. Fortunately, $\eta$ is preserved in this limit and hence Theorem \ref{thm:1982} holds. 

\begin{rem}\label{rem:eta}
	Theorem \ref{thm:1982} is less informative in case $\eta=0$, because of the flat connection on the trivial bundle.  On the contrary, the gluing method of Taubes applies only when $\eta=0$.
\end{rem}

In this paper, we consider the special case $G=SO(3)$ where $\pi_1(G)=\mathbb Z_2$. By exploiting an explicit form of an ASD $SO(3)$-connection on $\Real^4$ and some careful construction of test connections, we are able to draw more conclusions from a direct minimizing argument. More precisely, when $G=SO(3)$, the cohomology class $\eta(P)$ is the same as the second Stiefel-Whitney class $w_2(P)\in H^2(M,\mathbb Z_2)$ and
\begin{equation}
w_2(P)^2 = p_1(P) \qquad \text{mod} \quad 4.
	\label{eqn:w2p1}
\end{equation}
Here $w_2(P)^2$ is the Pontryagin square and lies in $H^4(M,\mathbb Z_4)$. Moreover, for each pair $(\eta,p_1)$ satisfying \eqref{eqn:w2p1} there is an $SO(3)$-bundle $E_{\eta,p_1}$ whose second Stiefel-Whitney class is $\eta$ and whose first Pontryagin class is $p_1$. Since $M$ is oriented, there is a natural isomorphism between $H^4(M,\mathbb Z)$ and $\mathbb Z$. Via this isomorphism, we may take $p_1$ as an integer. By \eqref{eqn:w2p1}, all possible values of $p_1$ form a set of the form $\set{4k+l}_{k\in \mathbb Z}$ for some $l\in \{0,1,2,3\}$. We denote this set by $K_\eta$.

Our main result is
\begin{thm}
	\label{thm:main}
	Suppose that $G=SO(3)$ and that $M$ is a closed oriented $4$-manifold with a Riemannian metric $g$. Fix an $\eta\in H^2(M,\mathbb Z_2)$. Then

	(1) for any $p_1\in K_\eta$ satisfying $-3\leq p_1\leq 3$, there is a smooth connection $D$ on $E_{\eta,p_1}$ which minimizes $\mathcal Y \mathcal M$ in $\mathcal C_{\eta,p_1}$.

	(2) for a fixed $p_1\geq 0$ in $K_\eta$, either there is a smooth connection $D$ on $E_{\eta,p_1}$ minimizing $\mathcal Y \mathcal M$ in $\mathcal C_{\eta,p_1}$; or there exists $p_{1,\infty}\in K_\eta$ with $p_1>p_{1,\infty}\geq 0$ such that $E_{\eta,p_1}$ admits an SD connection.

	(3) for a fixed $p_1\leq 0$ in $K_\eta$, either there is a smooth connection $D$ on $E_{\eta,p_1}$ minimizing $\mathcal Y \mathcal M$ in $\mathcal C_{\eta,p_1}$; or there exists $p_{1,\infty}\in K_\eta$ with $p_1<p_{1,\infty}\leq 0$ such that $E_{\eta,p_1}$ admits an ASD connection.
\end{thm}

\begin{rem}
	For underlying manifolds and bundles satisfying certain properties, we can be more precise in the part (2) and (3) of above theorem.

	Recall that for generic metric, the dimension of the moduli space of irreducible ASD connections over $E_{\eta,p_1}$ is
\[
-2p_1 - 3(1-b_1(M)+ b_+(M)),
\]
where $b_1$ and $b_+$ are betti numbers(see Section 4.3 in \cite{donaldson1990}.) Hence for some four manifold $M$, there is no {\it irreducible} ASD connections for small $-p_1$ (compared to a constant determined by $b_1(M)$ and $b_+(M)$). On the other hand, for $E_{\eta,p_1}$ to be reducible, there should exist a line bundle $L$ whose first Chern class $c_1\in H^2(M,\mathbb Z)$ satisfies
\[
p_1= c_1^2 \qquad \text{and} \qquad \eta\equiv c_1 \quad \text{mod} \quad 2.
\]
Hence, there are choices of $\eta$ in $H^2(M,\mathbb Z_2)$ such that $E_{\eta,p_1}$ can not be reducible if $\abs{p_1}$ is smaller than some constant depending on $\eta$ and the intersection form of $M$. For such bundles, the part (3) of the above theorem implies that there is a smooth minimizer in $\mathcal C_{\eta,p_1}$.
\end{rem}

The proof of Theorem \ref{thm:main} is by direct minimizing. The first part of it says that if $-3\leq p_1\leq 3$, then a minimizing sequence in $\mathcal C_{\eta,p_1}$ converges strongly and gives the minimizer that we want. The proof of this part amounts to ruling out the possible energy concentration. The second part shows that if we do minimizing in $C_{\eta,p_1}$, either the limit is strong and we get a minimizer in this class, or the concentration of energy results in the existence of a SD connection. Obviously, the third part is nothing but the orientation-reversing version of the second part.

The analytic framework of direct minimizing in \cite{sedlacek1982direct} is not strong enough for bubbling analysis. For our purpose, we pick an arbitrary minimizing sequence $D_i$ and mollify them using the Yang-Mills $\alpha$-flow studied in \cite{hong2015yang}. The advantage is that for the new sequence $D_i'$, we have strong estimates in neighborhoods where there is no energy concentration(see Lemma \ref{lem:46}). We briefly recall this analysis and its consequences in Section \ref{sec:pre}. If there is energy concentration, we obtain bubbles. In our case, the bubbles are energy minimizing connections of $SO(3)$-bundles over $S^4$. Notice that $SU(2)$ is the universal covering of $SO(3)$ and that $H^2(S^4,\mathbb Z)$ vanishes, hence there is explicit correspondence between $SU(2)$-connections and $SO(3)$-connections. On the other hand, for a degree one $SU(2)$-bundle over $S^4$, there is a well known SD connection due to 't Hooft (see Section \ref{subsec:bubble}). 

By using this explicit solution and a careful construction of test connections, we are able to prove
\begin{thm}\label{thm:tech}
	Suppose that $D$ is a smooth connection on $E_{\eta,p_1}$ which is not ASD, i.e. for some $\bx \in M$,
\[
(F_D)^+(\bx)\ne 0.
\]
Then for any $k\in \mathbb N$, there is a connection $D'$ on $E_{\eta,p_1-4k}$ such that
\[
\mathcal Y \mathcal M(D')< \mathcal Y \mathcal M(D) +  16k \pi^2.
\]
\end{thm}
This theorem is the main result of this paper from a technical point of view. Its proof relies on a construction similar to Section 7 and 8 of \cite{taubes1982self}, see also Section 7.2 of \cite{donaldson1990}. However, the computation here is more subtle. Instead of using one cut-off function to do the transition between two connections. We treat the first order approximation and the remainder differently by allowing overlaps for the first order approximation of the connections from both sides. It is this overlapping that contribute an interacting term in the expansion of energy functional. Due to the symmetry of Lie group, we can arrange so that the interacting term has a favorable sign from which Theorem \ref{thm:tech} follows.



The rest of the paper is organized as follows. In Section \ref{sec:pre}, we collect some known results on the Yang-Mills $\alpha$-flow and the explicit ASD connection. In Section \ref{sec:test}, we construct the test connection which leads to the proof of Theorem \ref{thm:tech}. In the final section, we prove Theorem \ref{thm:main}.

\section{Preliminaries}\label{sec:pre}

In this section, we collect a few known results and set up for the proofs.

\subsection{Yang-Mills $\alpha$-flow}\label{subsec:alpha}
We define
\[
\mathcal Y \mathcal M_\alpha(D)= \int_M (1+ \abs{F_A}^2)^\alpha dv
\]
as in \cite{hong2015yang}. For $\alpha>1$, it is shown there that the gradient flow of  $\mathcal Y \mathcal M_\alpha$ functional starting from an arbitrary initial connection exists for all time and converges to a critical point of $\mathcal Y \mathcal M_\alpha$. 

For a fixed $\eta$ and $p_1$, take $D_i\in \mathcal C_{\eta,p_1}$ satisfying
\[
\lim_{i\to \infty} \mathcal Y \mathcal M(D_i)= \inf_{D'\in \mathcal C_{\eta,p_1}} \mathcal Y \mathcal M(D').
\]
Since $D_i$ is smooth, we may take $\alpha_i$ sufficiently close to $1$ such that
\[
\mathcal Y \mathcal M_{\alpha_i}(D_i)\leq \mathcal Y \mathcal M(D_i) + V(M) + \frac{1}{i}.
\]
Here $V(M)$ is the volume of $M$. Let $D_i(t)$ be the solution of the $\alpha_i$-Yang-Mills flow from $D_i$ and denote $D_i(1)$ by $D_i'$. Then by the monotonicity of $\mathcal Y \mathcal M_{\alpha_i}$ energy along the flow, we have
\[
\mathcal Y \mathcal M(D_i') \leq \mathcal Y \mathcal M(D_i) + \frac{1}{i}.
\]
Hence, $D_i'$ is also a minimizing sequence. Moreover, we have
\begin{lem}[Lemma 4.6 in \cite{hong2015yang}]\label{lem:46}
There exists $\varepsilon>0$ such that if $B_r(x)\subset M$ satisfies
\[
\limsup_{i\to \infty} \int_{B_r(x)} \abs{F_{D_i'}}^2 dv \leq \varepsilon,
\]
then
\[
\norm{\nabla^k_{D_i'} F_{D_i'}}_{C^0(B_{r/4}(x))} \leq C r^{-k-2}
\]
for any $k\in \set{0}\cup \mathbb N$.
\end{lem}

With this lemma, we apply the routine bubbling analysis to the sequence $D_i'$. As a consequence, we obtained a limiting Yang-Mills connection, $D_\infty$ on some bundle $E_{\eta,p_{1,\infty}}$. Moreover, there exists $l$ $SO(3)$-bundle over $S^4$ (with round metric), denoted by $E_{p_{1,1}},\dots, E_{p_{1,l}}$, and Yang-Mills connection $D_{b;1},\dots,D_{b,l}$ such that

\begin{enumerate}[(M1)]
	\item $D_\infty$ is a minimizer of $\mathcal Y \mathcal M$ energy in the class $\mathcal C_{\eta,p_{1,\infty}}$;

	\item $D_{b,j}$ is a minimizer among all smooth connections of $E_{p_{1,j}}$ for $j=1,\dots,l$;

	\item the first Pontryagin number of $E_{p_{1,j}}$ is $p_{1,j}$, that of $E_{\eta,p_{1,\infty}}$ is $p_{1,\infty}$ and we have
		\begin{equation}
	p_1= p_{1,\infty} + \sum_{j=1}^l p_{1,j};
			\label{eqn:p1inf}
		\end{equation}
\item the energy identity holds (see Proposition 4.7 in \cite{hong2015yang})
	\begin{equation}
\lim_{i\to \infty} \mathcal Y \mathcal M(D_i)= \mathcal Y \mathcal M(D_\infty) + \sum_{j=1}^l \mathcal Y \mathcal M(D_{b,j}).
		\label{eqn:ei}
	\end{equation}
\end{enumerate}

Since $D_{b,j}$ minimizes $\mathcal Y \mathcal M$ function on $E_{p_{1,j}}$, we know
\[
\mathcal Y \mathcal M(D_{b,j})= 4\pi^2 \abs{p_{1,j}}
\]
for $j=1,\dots,l$. In fact, by ADHM construction, we know there admit ASD/SD connections over $E_{p_{1,j}}$, hence $\mathcal Y \mathcal M(D_{b,j})$ is the minimal possible value dictated by \eqref{eqn:pontryagin}. 

As a corollary, we have
\begin{lem}
	\label{lem:cor} All $p_{1,j}$'s are of the same sign with $p_1-p_{1,\infty}$.
\end{lem}
\begin{proof}
	If otherwise, \eqref{eqn:p1inf} implies that $\sum_{j=1}^l \abs{p_{1,j}}> \abs{\sum_{j=1}^lp_{1,j}}$. Let $D_b$ be an ASD(or SD) connection over
\[
E_{\sum_{j=1}^l p_{1,j}}:= \quad \text{the $SO(3)$-bundle over $S^4$ with Pontryagin number} \, \sum_{j=1}^l p_{1,j}, 
\]
whose existence is given by ADHM construction. Moreover, we have
\[
\mathcal Y \mathcal M(D_b)=4\pi^2 \abs{\sum_{j=1}^l p_{1,j}}.
\]
In case $p_1-p_{1,\infty}\ne 0$, by gluing $D_b$ and $D_\infty$ on the connected sum of $E_{\eta,p_{1,\infty}}$ and $E_{\sum_{j=1}^l p_{1,j}}$, for any $\varepsilon\in (0,1)$, we can find a connection $D'' \in \mathcal C_{\eta,p_1}$satisfying
\[
\mathcal Y \mathcal M(D'') \leq \mathcal Y \mathcal M(D_\infty) + 4\pi^2 \abs{\sum_{j=1}^l p_{1,j}} + \varepsilon < \mathcal Y \mathcal M(D_\infty)+ \sum_{j=1}^l \mathcal Y \mathcal M(D_{b,j}).
\]
This is a contradiction to \eqref{eqn:ei} and the fact that $D_i$ is a minimizing sequence in $\mathcal C_{\eta,p_1}$.
In case $p_1=p_{1,\infty}$, the left half of the above inequality holds by taking $D''=D_\infty$.
\end{proof}
\subsection{Standard bubble}\label{subsec:bubble}

On the $SU(2)$-bundle over $S^4$ with the second Chern class $c_2=1$, there is an ASD connection due to 't Hooft. Here we follow Section 3.4 of \cite{donaldson1990} for an explicit formulation of the connection.

By the conformal invariance of the problem, we write down the connection on $\Real^4$. Let $\mathbf{i}, \mathbf{j}$ and $\mathbf{k}$ be a standard basis of $\mathfrak s \mathfrak u(2)$. The connection one form is defined to be
\begin{equation}
	A= \frac{1}{1+\abs{x}^2} (\theta_1 \mathbf{i} + \theta_2 \mathbf{j} + \theta_3 \mathbf{k})
	\label{eqn:oldA}
\end{equation}
where
\begin{equation}
	\begin{split}
	\theta_1&= x_1 dx_2 - x_2dx_1 -x_3dx_4 + x_4 dx_3 \\
	\theta_2&= x_1 dx_3 - x_3dx_1 -x_4dx_2 + x_2 dx_4\\
	\theta_3&= x_1 dx_4 - x_4dx_1 -x_2dx_3 + x_3 dx_2.
	\end{split}
	\label{eqn:thetai}
\end{equation}
By direct computation,
\begin{equation}
	F= \left( \frac{1}{1+ \abs{x}^2} \right)^2 (d\theta_1 \mathbf{i} + d\theta_2 \mathbf{j} + d\theta_3 \mathbf{k}).
	\label{eqn:F}
\end{equation}
And it is obvious that $d\theta_i$'s are ASD $2$-forms.
In this paper, we are concerned with $SO(3)$-connections. With the isomorphism between $\mathfrak s \mathfrak o(3)$ and $\mathfrak s \mathfrak u(2)$, we may take the above as an ASD $SO(3)$-connection over $\Real^4$. It has $p_1=-4$ and the Yang-Mills energy equal to $16\pi^2$.

For our purpose, we would like to glue this standard bubble to a connection defined on some $4$-manifold $M$. For the definition in \eqref{eqn:oldA}, we used a global trivialization of the bundle over $\Real^4$. This trivialization does not agree (in the topological sense) with the one determined intrinsically by the curvature itself(near the infinity). To solve this problem, we consider a transformation on $\Real^4\setminus \set{0}$,
\begin{equation}
(x_1,x_2,x_3,x_4) \mapsto (\frac{x_1}{\abs{x}^2}, \frac{x_2}{\abs{x}^2}, \frac{x_4}{\abs{x}^2}, \frac{x_3}{\abs{x}^2}).
	\label{eqn:trans}
\end{equation}
Notice that we have purposedly reversed the order of $x_3$ and $x_4$ in order that the transformation keeps the orientation. By a pullback of \eqref{eqn:trans}, we obtain another way of looking at the same connection,
\begin{equation}
	A= \frac{1}{(1+ \abs{x}^2)\abs{x}^2}(\psi_1 \mathbf{i} + \psi_2 \mathbf{j} + \psi_3 \mathbf{k}),
	\label{eqn:newA}
\end{equation}
where
\begin{equation}
	\begin{split}
	\psi_1&= x_1 dx_2 - x_2dx_1 +x_3dx_4 - x_4 dx_3 \\
	\psi_2&= x_1 dx_3 - x_3dx_1 +x_4dx_2 - x_2 dx_4\\
	\psi_3&= x_1 dx_4 - x_4dx_1 +x_2dx_3 - x_3 dx_2.
	\end{split}
	\label{eqn:psii}
\end{equation}
There seems to be a singularity at $x=0$, but it is removable. 

In the rest of this paper, the notation $D_{stan}$ is used for the above connection on $\Real^4$ (with a removable singularity at the origin). Notice that we have chosen a trivialization on  $\set{\abs{x}>\delta}$ (for any $\delta>0$) in which the connection form is \eqref{eqn:newA}. If we use the removable singularity theorem again to regard $D_{stan}$ as a connection over $S^4$, then it lives on the bundle $E_{0,-4}(S^4)$ and has total energy $\mathcal Y \mathcal M(D_{stan})=16\pi^2$.

Next, we write $A$ in terms of the cylinder coordinates $(t,\omega)$, where $t\in \Real$ and $\omega \in S^3$. By definition,
\[
t= \log \abs{x}; \qquad \omega= \frac{x}{\abs{x}}\in S^3.
\]

We first notice that
\[
\theta_i(\partial_t) = \psi_i(\partial_t)=0
\]
and we denote by $\theta'_i$ and $\psi'_i$ the restrictions of $\theta_i$ and $\psi_i$ on $S^3\subset \Real^4$. If $\tilde{\Pi}$ is the  map from $\Real^4\setminus \set{0}$ to $S^3$ given by
\[
\tilde{\Pi}(t,\omega)=\omega,
\]
then by some abuse of notation, we also denote by $\theta'_i$ and $\psi'_i$ respectively
\[
\tilde{\Pi}^* \theta_i;\qquad \tilde{\Pi}^* \psi_i.
\]
If $T_s$ is the translating map on the cylinder, i.e. $T_s(t,\omega)=(t+s,\omega)$, then $\theta'_i$ and $\psi'_i$ are independent of $t$ in the sense that $(T_s)^* \theta'_i=\theta'_i$ and $(T_s)^*\psi'_i=\psi'_i$.

Moreover, if $\Pi$ is the coordinate change map given by
\[
\Pi(t,\omega)= e^t \omega,
\]
then 
\[
\Pi^*(\theta_i)=e^{2t} \theta'_i;\qquad \Pi^*(\psi_i)= e^{2t}\psi'_i.
\]
By a Taylor expansion, we obtain from \eqref{eqn:newA}
\begin{equation}
	\Pi^*(A)= e^{-2t} ( \psi_1' {\mathbf i} + \psi_2' {\mathbf j} + \psi_3' {\mathbf k}) + O(e^{-4t})
	\label{eqn:rightA}
\end{equation}
when $t\to \infty$.
This is the form of $D_{stan}$ that will be used in Section \ref{sec:test}.

\subsection{SD/ASD forms on cylinder}
In this section, we collect a few elementary computations which will be useful later. In the following lemma and the rest of this paper, we assume that the cylinder is oriented so that the map $\Pi$ is orientation preserving.
\begin{lem}\label{lem:twoforms}
	As two forms on cylinder, 
\[
d(e^{2t} \theta'_i)\quad \text{and} \quad d(e^{-2t}\psi_i') \quad \text{are ASD;}
\]
and
\[
d(e^{2t} \psi'_i)\quad \text{and} \quad d(e^{-2t}\theta_i') \quad \text{are SD.}
\]
\end{lem}
\begin{proof}
	Since the map $\Pi$ is conformal, its tangent map is a scaling between tangent spaces. Hence, it pulls back SD forms to SD forms. By \eqref{eqn:thetai} and \eqref{eqn:psii}, it is straightforward to check that $d\theta_i$'s are ASD forms on $\Real^4$ and $d\psi_i$'s are SD forms.

	The other half of the claim follows from the observation that the map $(t,\omega)\mapsto (-t,\omega)$ is an isometry and orientation-reversing.
\end{proof}

For fixed $(t,\omega)$, the space of ASD two forms, $\Lambda_-^2$ is a real vector space of dimension $3$. So is the space of SD two forms $\Lambda_+^2$. It will be clear in a minute that both $\set{d(e^{2t}\theta'_i)}_{i=1,2,3}$ and $\set{d(e^{-2t}\psi'_i)}_{i=1,2,3}$ are two orthogonal bases of $\Lambda_-^2$. The transition matrix between these two bases can be computed explicitly and it satisfies certain property that we shall need later.

\begin{lem}
	\label{lem:twist} Direct computation shows
	\begin{equation}
\left( 
\begin{array}[]{c}
	d(e^{-2t}\psi'_1) \\
	d(e^{-2t}\psi'_2) \\
	d(e^{-2t}\psi'_3) \\
\end{array}
\right)		
=
e^{-4t}T
\left( 
\begin{array}[]{c}
	d(e^{2t}\theta'_1) \\
	d(e^{2t}\theta'_2) \\
	d(e^{2t}\theta'_3) \\
\end{array}
\right)
		\label{eqn:twist}
	\end{equation}
where
\begin{equation}
T=-2	
\left( 
\begin{array}[]{ccc}
	\frac{x_1^2+x_2^2-x_3^2-x_4^2}{2r^2} & \frac{x_2x_3-x_1x_4}{r^2} & \frac{x_1x_3+x_2x_4}{r^2} \\
	\frac{x_1x_4+x_2x_3}{r^2} & \frac{x_1^2+x_3^2-x_2^2-x_4^2}{2r^2} & \frac{x_3x_4-x_1x_2}{r^2} \\
	\frac{x_2x_4-x_1x_3}{r^2} & \frac{x_1x_2+x_3x_4}{r^2} & \frac{x_1^2+x_4^2-x_2^2-x_3^2}{2r^2}
\end{array}
\right).
	\label{eqn:T}
\end{equation}
In particular, for any fixed $t$,
\[
\int_{S^3} T d\omega=0.
\]
Here $d\omega$ is the volume form of the round metric on $S^3$.
\end{lem}

%
%
\begin{proof}
The computation is easier in the Euclidean coordinates. By definition,
\begin{equation}
	\begin{split}
	d (e^{2t} \theta_1')=& 2 \Pi^* \left( dx_1\wedge dx_2 -dx_3\wedge dx_4 \right) \\
	d (e^{2t} \theta_2')=& 2 \Pi^* \left( dx_1\wedge dx_3 +dx_2\wedge dx_4 \right) \\
	d (e^{2t} \theta_3')=& 2 \Pi^* \left( dx_1\wedge dx_4 -dx_2\wedge dx_3 \right). 
	\end{split}
	\label{eqn:easier}
\end{equation}
\begin{rem}
	The orthogonality of $d(e^{2t}\theta'_1)$ and $d(e^{2t}\theta'_2)$ follows from that of $(dx_1\wedge dx_2-dx_3\wedge dx_4)$ and $(dx_1\wedge dx_3 + dx_2\wedge dx_4)$, which can be verified directly.	
\end{rem}
Exploiting the fact that
\[
d(e^{2t}\psi_1')= 2\Pi^* (dx_1\wedge dx_2 + dx_3\wedge dx_4),
\]
we compute
\begin{eqnarray*}
	d(e^{-2t}\psi_1')&=&  d\left( e^{-4t}\cdot e^{2t}\psi'_1 \right) \\
	&=& -4 e^{-4t} dt \wedge (e^{2t}\psi'_1) + e^{-4t} d \left( e^{2t}\psi'_1 \right)\\
	&=& -4 e^{-4t} \Pi^*\left( \frac{dr}{r}\wedge \psi_1 \right) +  2e^{-4t} \Pi^*(dx_1\wedge dx_2 + dx_3\wedge dx_4).\\
	&=& -4 e^{-6t} \Pi^*\left( (rdr)\wedge \psi_1 \right) +2 e^{-4t} \Pi^*(dx_1\wedge dx_2 + dx_3\wedge dx_4).\\
\end{eqnarray*}
Direct computation gives
\begin{eqnarray*}
(rdr)\wedge \psi_1 &=& (x_1dx_1+x_2dx_2+ x_3 dx_3+ x_4 dx_4)\wedge (x_1 dx_2 -x_2 dx_1 + x_3 dx_4-x_4dx_3) \\
&=& (x_1^2+x_2^2)(dx_1\wedge dx_2) + (x_3^2+ x_4^2) (dx_3\wedge dx_4) \\
&& + (x_2x_3-x_1x_4) (dx_1\wedge dx_3 + dx_2\wedge dx_4) +(x_1x_3+x_2x_4) (dx_1\wedge dx_4-dx_2\wedge dx_3),
\end{eqnarray*}
which implies that
\begin{eqnarray*}
	-\frac{1}{4} e^{4t} d(e^{-2t}\psi_1')&=& \frac{x_1^2+x_2^2-x_3^2-x_4^2}{2r^2} \Pi^*(dx_1\wedge dx_2 - dx_3\wedge dx_4) \\
	&&+ \frac{x_2x_3-x_1x_4}{r^2} \Pi^*(dx_1\wedge dx_3 + dx_2\wedge dx_4)\\
	&&+\frac{x_1x_3+x_2x_4}{r^2} \Pi^*(dx_1\wedge dx_4-dx_2\wedge dx_3).
\end{eqnarray*}
Similar computation yields
\begin{eqnarray*}
	-\frac{1}{4} e^{4t} d(e^{-2t}\psi_2')&=& \frac{x_1 x_4+x_2x_3}{r^2} (dx_1\wedge dx_2 - dx_3\wedge dx_4) \\
	&&+ \frac{x_1^2+x_3^2-x_2^2-x_4^2}{2r^2} (dx_1\wedge dx_3 + dx_2\wedge dx_4)\\
	&&+\frac{x_3x_4-x_1x_2}{r^2} (dx_1\wedge dx_4-dx_2\wedge dx_3)
\end{eqnarray*}
and
\begin{eqnarray*}
	-\frac{1}{4} e^{4t} d(e^{-2t}\psi_3')&=& \frac{x_2 x_4-x_1x_3}{r^2} (dx_1\wedge dx_2 - dx_3\wedge dx_4) \\
	&&+ \frac{x_1x_2+x_3x_4}{r^2} (dx_1\wedge dx_3 + dx_2\wedge dx_4)\\
	&&+\frac{x_1^2+x_4^2-x_2^2-x_3^2}{2r^2} (dx_1\wedge dx_4-dx_2\wedge dx_3).
\end{eqnarray*}
The final assertion about the integration of $T$ is trivial by symmetry.
\end{proof}

\subsection{Conformal normal coordinates}\label{subsec:normal}
Computations in the previous two subsections are valid on the standard metric of $\Real^4$ only. Notice that the concept of ASD/SD depends on the conformal class of the metric. Moreover, the problem of looking for Yang-Mills connection relies on the conformal class of the metric $g$ on $M$. 

For future use, we recall the existence and properties of the conformal normal coordinates.

\begin{thm}\label{thm:conformal}
	(Conformal normal coordinates). Let $M$ be a Riemannian manifold and $\bx\in M$. There is a conformal metric $g$ on $M$ such that
	\begin{equation}
		\det g_{ij}=1 + O(\abs{x}^3)
		\label{eqn:det}
	\end{equation}
	and
	\begin{equation}
		Ric(0)=0,
		\label{eqn:ric}
	\end{equation}
	where $x$ is the normal coordinates at $\bx$ with respect to $g$ and $Ric$ is the Ricci curvature.
	Moreover, there is the expansion
	\begin{equation}
		g_{pq}(x)= \delta_{pq} + \frac{1}{3} R_{pijq} x_ix_j + O(\abs{x}^3),
		\label{eqn:expg}
	\end{equation}
where $R_{pijq}$ is the Riemannian curvature tensor at $\bx$.
\end{thm}
\begin{rem}
	In this paper, we adopt the summation convention that repeated indices are summed.
\end{rem}

For a proof, we refer to Theorem 5.1 of \cite{parker1987} and the proof therein.

Consider the scaling map $S_\lambda:\Real^4\to \Real^4$ defined by
\[
S_\lambda (x)= \frac{x}{\lambda}.
\]
For two positive numbers $\lambda$ and $\delta$ ($\lambda << \delta$), let $A$ be given in \eqref{eqn:newA}, then
\[
(S_{\lambda})^* A 
\]
defines a connection on $B_{\lambda \delta^{-1}}$, which we denote by $D_{stan;\lambda}$. In the rest of this section, we use Theorem \ref{thm:conformal} to compare the Yang-Mills energy of $D_{stan,\lambda}$ measured with metric $g$ and with the flat metric $g_e$. 

In what follows, we denote by $\mathcal Y \mathcal M(D,\Omega)$ the Yang-Mills energy of $D$ restricted to the domain $\Omega$. In case we want to emphasize the metric $g$ of the underline manifold, we use a subscript $g$.

\begin{lem}
	\label{lem:conformal}
	We have
	\begin{equation}
		\mathcal Y \mathcal M_g(D_{stan,\lambda},B_{\lambda\delta^{-1}}) = \mathcal Y \mathcal M_{g_e}(D_{stan},B_{\delta^{-1}})+ O(\lambda^3)
		\label{eqn:ymg}
	\end{equation}
	and
	\begin{equation}
		\mathcal Y \mathcal M_g(D_{stan,\lambda}, B_{\delta}\setminus B_{\lambda \delta^{-1}}) = \mathcal Y \mathcal M_{g_e}(D_{stan}, B_{\lambda^{-1} \delta}\setminus B_{\delta^{-1}}) + O(\lambda^3).	
		\label{eqn:ymneck}
	\end{equation}
\end{lem}

\begin{rem}
	In comparison with \eqref{eqn:expg}, the $O(\lambda^3)$ above is unexpected. Indeed, this is why we need conformal normal coordinates, in particular, \eqref{eqn:ric}. 
\end{rem}

\begin{proof}
	Notice that $S_{\lambda^{-1}}$ maps $B_{\delta^{-1}}$ to $B_{\lambda \delta^{-1}}$. Set
\[
g_\lambda= \lambda^{-2} (S_{\lambda^{-1}})^* g.
\]

For the proof of \eqref{eqn:ymg}, we use \eqref{eqn:expg} to derive the following expansion
\[
(g_\lambda)_{pq}= \delta_{pq} + \frac{1}{3} \lambda^2 R_{pijq} x_ix_j + O(\lambda^3).
\]
This holds uniformly for $x\in B_{\delta^{-1}}$ in the sense that the constant in the definition of $O(\lambda^3)$ depends on $\delta$.

Hence,
\begin{equation}
(g_\lambda)^{pq}= \delta_{pq} - \frac{1}{3} \lambda^2 R_{pijq} x_ix_j + O(\lambda^3).
	\label{eqn:gpq}
\end{equation}
By the scaling invariance of Yang-Mills energy, we have
\begin{eqnarray*}
	&&\mathcal Y \mathcal M_g(D_{stan,\lambda},B_{\lambda \delta^{-1}}) \\
	&=& \mathcal Y \mathcal M_{g_\lambda} (D_{stan},B_{\delta^{-1}}) \\
	&=& \int_{B_{\delta^{-1}}} \left( F_{D_{stan}}, F_{D_{stan}} \right)_{\Lambda^2(g_\lambda)\otimes \mathfrak g} dV_{g_\lambda}\\
	&=& \mathcal Y \mathcal M_{g_e}(D_{stan}, B_{\delta^{-1}}) + \int_{B_{\delta^{-1}}} (-\frac{2}{3} \lambda^2 R_{pijq}x_ix_j) (F_{ps}, F_{qs})_{\mathfrak g} dx + O(\lambda^3).
\end{eqnarray*}
Here we have used \eqref{eqn:gpq} and \eqref{eqn:det}. Hence it remains to show that the second term in the above line vanishes. We shall verify that for each point $x\in B_{\delta^{-1}}$, the integrand is zero. More precisely, we claim that for any fixed $x$,
\begin{equation}
	R_{pijq}(F_{ps},F_{qs})_{\mathfrak g}=0.
	\label{eqn:claim}
\end{equation}
This requires a special property of $D_{stan}$. 

By \eqref{eqn:F}, this $F$, regarded as a linear map from $\Lambda^2_-$ to $\mathfrak s \mathfrak o(3)$, maps an orthogonal basis of $\Lambda^2_-$ to an orthogonal basis of $\mathfrak s \mathfrak o(3)$. More precisely, it maps $d\theta_1$ to a constant multiple ${\mathbf i}$ and so on. This property is preserved when we pull back the connection by an orientation-preserving conformal map. Indeed, its induced map from $\Lambda^2_-$ to $\Lambda^2_-$ is a composition of scaling and rotation. Since \eqref{eqn:newA} is obtained from \eqref{eqn:oldA} by such a pullback, we may assume 
\[
F_{D_{stan}}= \eta(x) \left( \omega_1 \tilde{\mathbf i} + \omega_2 \tilde{\mathbf j} + \omega_3 \tilde{\mathbf k} \right)
\]
where $(\tilde{\mathbf i},\tilde{\mathbf j}, \tilde{\mathbf k})$ is an orthogonal basis of $\mathfrak s \mathfrak o(3)$ and $(\omega_1,\omega_2,\omega_3)$ is an orthogonal basis of $\Lambda^2_-$ (measured with $g_e$). We may rotate the normal coordinates by an action of $SO(4)$. Since the induced action on $\Lambda^2_-$ could be any $SO(3)$, we may assume that $\omega_j=\pm d\theta_j$ for $j=1,2,3$. If we set $d\theta_k= G^{(k)}_{ij}dx_i\wedge dx_j$, then the matrices $G^{(k)}_{ij}$ are (from $k=1$ to $3$)
\[
\left( 
\begin{array}[]{cccc}
	0 & 1 & 0 & 0 \\
	-1 & 0 & 0 & 0 \\
	0 & 0 & 0 & -1 \\
	0 & 0 & 1 & 0
\end{array}
\right),
\left( 
\begin{array}[]{cccc}
	0 & 0 & 1 & 0 \\
	0 & 0 & 0 & 1 \\
	-1 & 0 & 0 & 0 \\
	0 & -1 & 0 & 0
\end{array}
\right),
\left( 
\begin{array}[]{cccc}
	0 & 0 & 0 & 1 \\
	0 & 0 & -1 & 0 \\
	0 & 1 & 0 & 0 \\
	-1 & 0 & 0 & 0
\end{array}
\right).
\]
It is straightforward to check that for any $k$
\[
G^{(k)}_{ps} G^{(k)}_{qs}= \delta_{pq}.
\]
With this observation, we may compute explicitly
\begin{eqnarray*}
	(F_{ps},F_{qs})_{\mathfrak g}= c \eta(x)^2 \delta_{pq}
\end{eqnarray*}
for some constant $c$.
Hence, the proof of \eqref{eqn:ymg} is completed by noticing that $R_{pijq} \delta_{pq}=0$ due to \eqref{eqn:ric}.

By applying the pullback of \eqref{eqn:trans} to \eqref{eqn:F}, we may estimate (for some universal $c$)
\begin{equation*}
	\abs{F_{D_{stan}}} \leq  \frac{c}{\abs{x}^4}
\end{equation*}
for $\abs{x}\geq \delta^{-1}$. Hence,
\begin{equation}
	\abs{F_{D_{stan,\lambda}}} \leq \frac{c\lambda^2}{ \abs{x}^4}
	\label{eqn:Fstan}
\end{equation}
for $\abs{x}\geq \lambda \delta^{-1}$.

Hence, we may compute
\begin{eqnarray*}
	&& \mathcal Y \mathcal M_{g} (D_{stan,\lambda}, B_\delta\setminus B_{\lambda \delta^{-1}}) - \mathcal Y \mathcal M_{g_e} (D_{stan}, B_{\lambda^{-1}\delta}\setminus B_{\delta^{-1}}) \\
	 &=& \mathcal Y \mathcal M_{g} (D_{stan,\lambda}, B_\delta\setminus B_{\lambda \delta^{-1}}) - \mathcal Y \mathcal M_{g_e} (D_{stan,\lambda}, B_\delta\setminus B_{\lambda \delta^{-1}}) \\
	 &=& \int_{B_\delta\setminus B_{\lambda \delta^{-1}}} (F_{D_{stan,\lambda}}, F_{D_{stan,\lambda}})_{\Lambda^2(g)\otimes \mathfrak g} dV_g - \int_{B_\delta\setminus B_{\lambda \delta^{-1}}} (F_{D_{stan,\lambda}}, F_{D_{stan,\lambda}})_{\Lambda^2(g_e)\otimes \mathfrak g} dV_{g_e}\\
	 &=& \int_{B_\delta\setminus B_{\lambda \delta^{-1}}} (F_{D_{stan,\lambda}}, F_{D_{stan,\lambda}})_{\Lambda^2(g)\otimes \mathfrak g} (dV_g-dV_{g_e}) \\
	 && + \int_{B_\delta\setminus B_{\lambda \delta^{-1}}} (F_{D_{stan,\lambda}}, F_{D_{stan,\lambda}})_{\Lambda^2(g)\otimes \mathfrak g} - (F_{D_{stan,\lambda}}, F_{D_{stan,\lambda}})_{\Lambda^2(g_e)\otimes \mathfrak g} dV_{g_e}\\
	 &:=& I_1 + I_2.
\end{eqnarray*}
By \eqref{eqn:Fstan} and \eqref{eqn:det}, we get
\begin{equation}
\abs{I_1}\leq C(\delta) \int_{B_\delta\setminus B_{\lambda \delta^{-1}}} \frac{\lambda^4}{ \abs{x}^8} \abs{x}^3 dV_{g_e} \leq C(\delta) \lambda^3.
	\label{eqn:i1}
\end{equation}
If we set $h_{pq}(x)=-\frac{1}{3} R_{pijq}x_ix_j$, we get from \eqref{eqn:expg}
\[
g^{pq}(x)= \delta_{pq} +h_{pq}(x) + O(\abs{x}^3).
\]
We use the bilinearity of $(F_{D_{stan,\lambda},F_{D_{stan,\lambda}}})_{\Lambda^2(g)\otimes \mathfrak g}$ with respect to the inverse of $g$ and the fact that
\[
\abs{g^{pq}-\delta_{pq}}\leq C \abs{x}^2
\]
to obtain
\begin{eqnarray*}
	I_2&=&  2 \int_{B_\delta\setminus B_{\lambda \delta^{-1}}} h_{pq} (F_{ps}, F_{qs})_{\mathfrak g} dV_{g_e} + \int_{B_\delta\setminus B_{\lambda \delta^{-1}}} O(\abs{x}^3) \abs{F_{D_{stan,\lambda}}}^2_{g_e} dV_{g_e}.
\end{eqnarray*}
The first term in the above equation vanishes because of the same proof of \eqref{eqn:claim} in the previous part of the proof, and the second term is estimated as $I_1$ in \eqref{eqn:i1}. Hence,
\[
\abs{I_2}\leq C(\delta) \lambda^3.
\]
The proof is done.
\end{proof}

\section{Construction of test connection}\label{sec:test}

Let $M$ be a closed oriented $4$-manifold and $E_{\eta,p_1}$ is some $SO(3)$-bundle. Let $D$ be a smooth connection in $\mathcal C_{\eta,p_1}$. Assume that there is some ${\bf x}_0\in M$ such that
\[
	(F_D)^+({\bf x}_0)\ne 0.
\]

The aim of this section is to show that by gluing a standard bubble at ${\bf x}_0$, we obtain a connection $D'$ on the bundle $E_{\eta,p_1-4}$ such that
\begin{equation}
	\mathcal Y \mathcal M(D')< \mathcal Y \mathcal M(D) + 16\pi^2.
	\label{eqn:goal}
\end{equation}
It will be clear in the construction below that for some small $\delta>0$, $D'$ and $D$ restricted to $M\setminus B_{\delta}(\bx_0)$ are identical. Hence, Theorem \ref{thm:tech} follows by taking $k$ different points $\bx_l$ ($l=1,\dots,k$) with $(F_D)^+(\bx_l)\ne 0$, taking small $\delta$ so that $B_\delta(\bx_l)$'s are disjoint and repeating the construction in each $B_\delta(\bx_l)$.

\subsection{Preparation}
Take a {\bf positively oriented} conformal normal coordinate system $(x_1,\dots,x_4)$ around ${\bf x}_0$ (as in Section \ref{subsec:normal}). Fix an orthonormal frame $\set{e_1,e_2,e_3}$ on the fiber over ${\bf x}_0$ and extend it to a neighborhood of ${\bf x}_0$ by parallel transportation along the curve
\[
	t\mapsto (tx_1,tx_2,tx_3,tx_4).
\]
Using this local frame, the connection $D$ is represented by a matrix valued one form
\[
	A= A_i(x) dx_i.
\]
By definition, we have $A_i(0)=0$ and therefore there exist constant matrices $C_{ij}$ such that
\begin{equation}
	A_{i}(x)= C_{ij} x_j  + O(\abs{x}^2).
	\label{eqn:expandA}
\end{equation}
By the formula
\[
F_{ij}=\pfrac{A_j}{x_i}- \pfrac{A_i}{x_j} + [A_i,A_j],
\]
we have
\[
F_{ij}(x)= \left( C_{ji}-C_{ij} \right) + O(\abs{x}).
\]
Hence,
\[
F_{ij}(0)= C_{ji}-C_{ij}.
\]
On the other hand, by the definition of local frame again, $A(\partial_r)=0$, where $r^2=\sum_{i=1}^4 x_i^2$, i.e.
\[
C_{kj} x_k x_j=0.
\]
This means that $C$ is skew-symmetric and hence,
\[
A= \frac{1}{2} F_{ji}(0)x_j  dx_i + O(\abs{x}^2)dx_* = \frac{1}{2}\sum_{i<j} F_{ij}(x_i dx_j -x_jdx_i) + O(\abs{x}^2) dx_*.
\]
The leading term in the expansion of $A$ is responsible for the curvature form at ${\bf x}_0$.

To introduce the SD and ASD decomposition, recall the definitions of $\theta_i$ and $\psi_i$ in \eqref{eqn:thetai} and \eqref{eqn:psii}. Obviously,  with respect to the flat metric $g_e$, $d\theta_i$ are ASD $2$-forms and $d\psi_i$ are SD $2$-forms.
By noticing simple algebra relation such as
\begin{equation*}
	(x_1 dx_2 - x_2 dx_1)= \frac{1}{2}(\theta_1+\psi_1) \quad \text{etc,}
\end{equation*}
we obtain a decomposition
\[
A=\sum_{i=1,2,3} \left( F_{-,i}\theta_i + F_{+,i}\psi_i\right) + O(\abs{x}^2)dx_*,
\]
where $F_{-,i}$ and $F_{+,i}$ are constant matrices. Moreover, the curvature $F_D$ at ${\bf x}_0$ is SD(ASD) if and only if $F_{-,i}$($F_{+,i}$) vanishes.

Finally, recalling the definition of $\Pi(t,\omega)=e^t \omega$, we pullback $A$ to be a one form defined on cylinder
\begin{equation}
\Pi^*(A)= \sum_{i=1,2,3} \left( F_{-,i} e^{2t}\theta'_i + F_{+,i} e^{2t}\psi'_i \right) + O(e^{3t}).
	\label{eqn:leftA}
\end{equation}

\subsection{Gluing} \label{subsec:glue}

We need to introduce two parameters: $\delta>0$ and $\lambda>0$.  The first one describes the size of the neighborhood where the gluing happens. For the connection $D$, the gluing occurs inside the ball $B_\delta(0)$ (in terms of the normal coordinates chosen in the previous subsection); for the standard bubble, the gluing occurs outside the ball $B_{\delta^{-1}}(0)$ in the $x$ coordinates. The second parameter $\lambda$ is a scaling factor and it is a lot smaller than $\delta$. We will scale down the standard bubble by $\lambda$ and attach it to the ball $B_{\lambda \delta^{-1}}(0)$ (normal coordinates again) on $M$. 

\begin{rem}
	If one takes this gluing as a reverse process of bubbling, then this $\lambda$ is usually understood as the scale of energy concentration.
\end{rem}

Before we proceed, we describe the setting of the computations that follows. We have a long cylinder $[\log \lambda -\log \delta, \log \delta]\times S^3$, to the left ($t=\log \delta$) of which we attach the manifold $M$ (with $B_\delta(0)$ removed) together with the bundle. To the right end ($t=\log \lambda -\log \delta$), we attach a scaled (by $\lambda$) version of standard bubble (in $x$ coordinates). Moreover, to be precise, we need to specify how we glue the bundles. On $B_\delta(0)\subset M$, we have chosen a trivialization by parallel transportation, in which \eqref{eqn:leftA} holds. On a neighborhood of infinity $(\abs{x}>\delta^{-1})$ of the standard bubble, there is a trivialization in which we have \eqref{eqn:rightA}. In the following computation, we will need to rotate the second trivialization by a constant element in $SO(3)$, before we identify the two trivializations. It turns out that we need this rotation to move the curvature form (at $x=\infty$) into a favorable position(see \eqref{eqn:P} and Lemma \ref{lem:trivial}). Keep in mind that in \eqref{eqn:rightA}, we have the freedom of choosing a standard basis of $\mathfrak s \mathfrak u(2)$.

To proceed, we rewrite \eqref{eqn:rightA} (after a scaling by parameter $\lambda$)
\begin{equation}
	A_{right}:=\lambda^2 e^{-2t} \left( \psi_1' {\mathbf i} +\psi_2' {\mathbf j} + \psi_3' {\mathbf k}\right) + w_r
	\label{eqn:Aright}
\end{equation}
where $w_r$ (as a one form on cylinder) satisfies
\begin{equation}
	\sup_{t>\log \lambda - \log \delta} \norm{w_r}_{C^2( [t,t+1]\times S^3)} \leq C \lambda^4 e^{-4t}.
	\label{eqn:Wright}
\end{equation}
On the other hand, \eqref{eqn:leftA} is rewritten as
\begin{equation}
	A_{left}:= \sum_{i=1,2,3} e^{2t}(F_{-,i}\theta_i' + F_{+,i} \psi_i') + w_l
	\label{eqn:Aleft}
\end{equation}
where
\begin{equation}
	\sup_{t<\log \delta} \norm{w_l}_{C^2( [t,t+1]\times S^3)}\leq C e^{3t}.
	\label{eqn:Wleft}
\end{equation}

To have a smooth transition, we need the cut-off functions $\varphi_1$, $\varphi_2$, $\varphi_3$ and $\varphi_4$ as follows. They are functions defined on the cylinder $[\log \lambda -\log \delta, \log \delta]\times S^3$ and they depend on $t$ alone.

(C1) $\varphi_1(t,\theta)$ is supposed to be $1$ for all $t> \frac{1}{2}\log \lambda+1$ and $0$ for all $t<\frac{1}{2}\log \lambda -1$. Let $\varphi_2=1-\varphi_1$;

(C2) let $\varphi_3(t,\theta)$ to be $1$ for all $t> \log \lambda -\log \delta +2$ and $0$ for $t<\log \lambda -\log \delta+1$;

(C3) let $\varphi_4(t,\theta)= \varphi_3(\log \lambda - t,\theta)$.

Finally, we set
\begin{equation}
	\begin{split}
		A_{(\lambda)}:=& \varphi_3(t) e^{2t}  \sum_{i=1,2,3} (F_{-,i}\theta_i' + F_{+,i} \psi_i')\\
		& + \varphi_4(t)\lambda^2 e^{-2t} \left( \psi_1' {\mathbf i} +\psi_2' {\mathbf j} + \psi_3' {\mathbf k} \right) \\
		& + \varphi_1(t) w_l + \varphi_2(t) w_r
	\end{split}
	\label{eqn:defA}
\end{equation}
and define a new connection $D'$, which is 
\begin{itemize}
	\item $D$ on $M\setminus B_\delta$;
	\item $A_{(\lambda)}$ (in the trivialization discussed in the beginning of Section \ref{subsec:glue}) on $B_\delta\setminus B_{\lambda \delta^{-1}}$;
	\item $D_{stan,\lambda}$ inside $B_{\lambda \delta^{-1}}$.
\end{itemize}

Since the ASD connection given in Section \ref{sec:pre} has Pontryagin number equal to $-4$, the new connection lives on $E_{\eta,p_1-4}$. It remains to compute its Yang-Mills energy and compare it with $\mathcal Y \mathcal M(D)+16\pi^2$ (see \eqref{eqn:goal}).

\subsection{Energy estimate}

Since the total energy of $D_{stan}$ on the flat $\Real^4$ is $16\pi^2$, we have
\begin{eqnarray*}
	&& \mathcal Y \mathcal M(D')- \mathcal Y \mathcal M(D) \\
	&=& \mathcal Y \mathcal M(D', B_\delta) - \mathcal Y \mathcal M(D, B_\delta) \\
	&=& \mathcal Y \mathcal M_g (D', B_\delta \setminus B_{\lambda \delta^{-1}}) + \mathcal Y \mathcal M_g(D_{stan,\lambda}, B_{\lambda \delta^{-1}}) - \mathcal Y \mathcal M(D,B_\delta) \\
	&=& E_{gain}- E_{loss} + 16 \pi^2 \\
	&& + \mathcal Y \mathcal M_g(D_{stan,\lambda}, B_{\lambda \delta^{-1}})- \mathcal Y \mathcal M_{g_e}(D_{stan}, B_{\delta^{-1}}),
\end{eqnarray*}
where
\[
E_{loss}:=\mathcal Y \mathcal M_g(D, B_\delta) + \mathcal Y \mathcal M_{g_e}(D_{stan}, \Real^4\setminus B_{\delta^{-1}})
\]
and
\[
	E_{gain}:=\mathcal Y \mathcal M_g(D', B_\delta\setminus B_{\lambda \delta^{-1}}).
\]

Lemma \ref{lem:conformal} implies that to prove \eqref{eqn:goal}, it suffices to prove
\begin{equation}
	E_{gain}-E_{loss}= C \lambda^2 + \varepsilon(\delta) \lambda^2
	\label{eqn:aim}
\end{equation}
for some {\bf negative} constant $C$, where $\varepsilon(\delta)$ is a small constant depending on $\delta$ satisfying $\lim_{\delta\to 0} \varepsilon(\delta)=0$.

\begin{rem}
	In what follows, we will use $\varepsilon(\delta)$ for other constants satisfying the same requirement.
\end{rem}

First, notice that $D$ and $D_{stan}$ are fixed (independent of $\lambda$), hence
\begin{equation}
	E_{loss}:=\mathcal Y \mathcal M_g(D, B_\delta\setminus B_{\lambda \delta^{-1}}) + \mathcal Y \mathcal M_{g_e}(D_{stan}, B_{\lambda^{-1} \delta}\setminus B_{\delta^{-1}}) + O(\lambda^4).
	\label{eqn:loss}
\end{equation}
Hence the comparison occurs on the neck domain $B_\delta \setminus B_{\lambda \delta^{-1}}$ and we shall work with the cylinder coordinates $(t,\omega)$. The Euclidean metric $g_e$ is conformal to the cylinder metric $g_{cy}:=dt^2+ d\omega^2$. We will write $d\omega dt$ for the volume form and use $\abs{\cdot}$ and $\langle \cdot, \cdot \rangle $ for the norm and inner product with respect to $g_{cy}$. In $(t,\omega)$ coordinates, the conformal metric $g_c:=\abs{x}^{-2}g$ has the form
\[
g_c=dt^2 + g_\omega(t)
\]
where $g_{\omega}(t)$ is a family of metric on $S^3$ that tends to the round metric on $S^3$ when $t\to -\infty$, i.e. for $t\leq \log \delta$,
\begin{equation}
	\abs{g_c- g_{cy}}_{g_{cy}}\leq C e^t.
	\label{eqn:smaller}
\end{equation}
We use $\abs{\cdot}_{g_c}$, $\langle \cdot, \cdot \rangle_{g_c}$ for the norm and inner product with respect to the metric $g_c$ and write $(d\omega dt)_{g_c}$ for its volume form.

\begin{rem}
	Notice that the norms and the volume forms of $g_c$ and those of the cylinder metric are comparable. Hence in many coarse estimates, there is no need to distinguish between them.
\end{rem}

The definition of \eqref{eqn:defA} naturally decomposes into two parts
\[
	A_{(\lambda)}= A_L + A_R
\]
where
\begin{equation}
	\begin{split}
		A_L:=& \varphi_3(t) e^{2t}  \sum_{i=1,2,3} (F_{-,i}\theta_i' + F_{+,i} \psi_i') + \varphi_1(t) w_l \\
		A_R:=& \varphi_4(t)\lambda^2 e^{-2t} \left( \psi_1' {\mathbf i} +\psi_2' {\mathbf j} + \psi_3' {\mathbf k} \right) + \varphi_2(t) w_r.
	\end{split}
	\label{eqn:Aleftright}
\end{equation}
Naturally, there is also a decomposition of $E_{gain}$,
\begin{eqnarray*}
	E_{gain}&=& \int_{[\log \lambda-\log \delta,\log \delta]\times S^3} \langle dA_{(\lambda)}+ [A_{(\lambda)}\wedge A_{(\lambda)}],  dA_{(\lambda)}+ [A_{(\lambda)}\wedge A_{(\lambda)}] \rangle_{g_c}  (d\omega dt)_{g_c} \\
	&=& E_{left}+E_{right}+E_{inter},
\end{eqnarray*}
where
\begin{equation*}
	\begin{split}
		E_{left}:=& \int_{[\log \lambda-\log \delta,\log \delta]\times S^3} \langle dA_L+ [A_L \wedge A_L],  dA_L+ [A_L \wedge A_L]  \rangle_{g_c}  (d\omega dt)_{g_c} \\
		E_{right}:=& \int_{[\log \lambda-\log \delta,\log \delta]\times S^3} \langle dA_R+ [A_R \wedge A_R],  dA_R+ [A_R \wedge A_R]  \rangle_{g_c}  (d\omega dt)_{g_c},
	\end{split}
\end{equation*}
and we used $E_{inter}$ to represent all terms related to the interaction between $A_L$ and $A_R$. 

\begin{lem}
	\label{lem:loss}
	\begin{equation}
		\begin{split}
			E_{left}-\mathcal Y \mathcal M_{g}(D,B_\delta\setminus B_{\lambda \delta^{-1}})&= O(\lambda^{5/2})\\	
			E_{right}-\mathcal Y \mathcal M_{g_e}(D_{stan},B_{\lambda^{-1} \delta}\setminus B_{\delta^{-1}})&= O(\lambda^3). 
		\end{split}
		\label{eqn:leftright}
	\end{equation}
\end{lem}
\begin{proof}
Set
\[
	Q=A_L -A_{left}
\]
where $A_{left}$ is the local expression of connection $D$(see \eqref{eqn:Aleft}).
By the definitions of $\varphi_3$ and $\varphi_1$ and \eqref{eqn:Wleft}, we have
\begin{eqnarray*}
	\abs{dQ}+\abs{Q}&\leq & C(\delta) \chi_{[\log \lambda-\log \delta, \log\lambda -\log \delta+2]}(t) \lambda^2 + C(\delta) \chi_{[\log \lambda - \log \delta, \frac{1}{2}\log \lambda +1]}(t) e^{3t},
\end{eqnarray*}
where $\chi_{[a,b]}(t)$ is $1$ if $t\in [a,b]$ and $0$ otherwise.

By the definition of $A_L$ and $A_{left}$, we have
\begin{equation}
	\abs{A_L}, \abs{dA_L}, \abs{A_{left}}, \abs{dA_{left}} \leq C e^{2t}.
	\label{eqn:boundleft}
\end{equation}
Hence, by subtracting $E_{left}$ with the definition of $\mathcal Y \mathcal M_g(D, B_\delta\setminus B_{\lambda \delta^{-1}})$ (in $(t,\omega)$ coordinates), we obtain
\begin{eqnarray*}
	&& \abs{E_{left}- \mathcal Y \mathcal M_g(D,B_\delta\setminus B_{\lambda \delta^{-1}})}  \\
	&\leq& \int_{[\log \lambda - \log \delta, \log \delta]} (\abs{dQ}+\abs{Q}) C(\delta) e^{2t} dt \\
	&\leq& C(\delta) \lambda^4 + C(\delta) \int_{\log \lambda}^{\frac{1}{2}\log \lambda} e^{5t} dt \\
	&\leq& C(\delta) \lambda^{5/2}.
\end{eqnarray*}

The proof for the second half of the lemma starts with a similar computation. Set
\[
	Q=A_R -A_{right}
\]
By the definitions of $\varphi_4$ and $\varphi_2$ and \eqref{eqn:Wright}, we have
\begin{eqnarray*}
	\abs{dQ}+\abs{Q}&\leq & C(\delta) \chi_{[\log \delta-2, \log \delta]}(t) \lambda^2 + C(\delta) \chi_{[\frac{1}{2}\log \lambda -1, \log \delta]}(t) \lambda^4 e^{-4t}.
\end{eqnarray*}
Since
\begin{equation}
	\abs{A_R}, \abs{dA_R}, \abs{A_{right}}, \abs{dA_{right}} \leq C \lambda^2 e^{-2t},
	\label{eqn:boundright}
\end{equation}
we obtain as before
\begin{eqnarray*}
	&& \abs{E_{right}- \mathcal Y \mathcal M_{g}(D_{stan,\lambda},B_{\delta}\setminus B_{\lambda \delta^{-1}})}  \\
	&\leq& \int_{[\log \lambda - \log \delta, \log \delta]} (\abs{dQ}+\abs{Q}) C(\delta) \lambda^2 e^{-2t} dt \\
	&\leq& C(\delta) \lambda^4 + C(\delta) \int_{\frac{1}{2}\log \lambda}^{\log \delta} \lambda^6 e^{-6t} dt \\
	&\leq& C(\delta) \lambda^{3}.
\end{eqnarray*}
The rest of the proof follows from the second half of Lemma \ref{lem:conformal}.
\end{proof}

Next, we consider $E_{inter}$. By definition, $E_{gain}$ is a linear combination of the integrals of
\[
\langle dA_{(\lambda)}, dA_{(\lambda)} \rangle_{g_c}, \langle dA_{(\lambda)}, [A_{(\lambda)}\wedge A_{(\lambda)}] \rangle_{g_c} , \langle [A_{(\lambda)}\wedge A_{(\lambda)}], [A_{(\lambda)}\wedge A_{(\lambda)}] \rangle_{g_c} .
\]
The three terms above are quadratic, cubic and quartic respectively. When we substitute $A_{(\lambda)}$ by $A_L+A_R$, we obtain a linear combination of the integrals of
\[
\langle dA_{*}, dA_{*} \rangle_{g_c}, \langle dA_{*}, [A_{*}\wedge A_{*}] \rangle_{g_c} , \langle [A_{*}\wedge A_{*}], [A_{*}\wedge A_{*}] \rangle_{g_c} 
\]
where $*$ stands for either 'L' or 'R'. The terms in which all $*$'s are 'L' (or 'R') have been considered in $E_{left}$ (or $E_{right}$ respectively). It remains to consider those terms in which both 'L' and 'R' appear. $E_{inter}$ is a linear combination of the integrals of such terms.

Due to \eqref{eqn:boundleft} and \eqref{eqn:boundright}, we estimate the size of cubic term by 
\[
\abs{\langle dA_L,[A_R\wedge A_R]\rangle_{g_c}}\leq C(\delta) \lambda^2 (\lambda^2 e^{-2t}),
\]
which implies after integrating over $[\log \lambda -\log \delta,\log \delta]\times S^3$
\begin{equation*}
	\abs{\int \langle dA_L, [A_R\wedge A_R] \rangle_{g_c} (d\omega dt)_{g_c}}\leq \varepsilon(\delta) \lambda^2.
\end{equation*}
Here $\varepsilon(\delta)$ is a constant depending on $\delta$ which satisfies
\[
\lim_{\delta\to 0} \varepsilon(\delta)=0.
\]
It arises when we compute the integral
\[
\int_{\log \lambda-\log \delta}^{\log \delta} \lambda^2 e^{-2t} dt.
\]
Similar arguments work for all cubic and quartic terms. In summary, we have proved that
\begin{equation}
	\abs{E_{inter}- \int_{[\log\lambda -\log \delta,\log \delta]\times S^3} \langle dA_R, dA_L \rangle_{g_c}  (d\omega dt)_{g_c}}\leq \varepsilon(\delta) \lambda^2.
	\label{eqn:inter3}
\end{equation}
Due to \eqref{eqn:boundleft} and \eqref{eqn:boundright}, we derive from \eqref{eqn:inter3}
\begin{equation}
	\abs{E_{inter}- \int_{[\log\lambda -\log \delta,\log \delta]\times S^3} \langle dA_R, dA_L \rangle  d\omega dt}\leq \varepsilon(\delta) \lambda^2.
	\label{eqn:inter}
\end{equation}
In fact, for any matrix valued two forms $V$ and $W$, it follows from \eqref{eqn:smaller} that 
\[
\abs{\langle V, W \rangle_{g_c}- \langle V, W \rangle }\leq C \abs{V}\abs{W} e^t 
\]
and
\[
\abs{(d\omega dt)_{g_c} - (d\omega dt)}\leq Ce^t (d\omega dt).
\] The small constant $\varepsilon(\delta)$ appears when we integrate over the cylinder as before.

By throwing away the higher order terms in \eqref{eqn:Aleftright}, we define
\begin{equation}
	\begin{split}
		A'_L:=& \varphi_3(t) e^{2t}  \sum_{i=1,2,3} (F_{-,i}\theta_i' + F_{+,i} \psi_i') \\
		A'_R:=& \varphi_4(t)\lambda^2 e^{-2t} \left( \psi_1' {\mathbf i} +\psi_2' {\mathbf j} + \psi_3' {\mathbf k} \right).
	\end{split}
	\label{eqn:Ano_w}
\end{equation}
Direct computation shows
\begin{equation}
	\begin{split}
		\langle dA_L, dA_R \rangle =& \langle dA'_L, dA'_R \rangle  \\
		& + \langle dA_L', d(\varphi_2(t) w_r)\rangle + \langle dA'_R, d(\varphi_1(t)w_l) \rangle \\
		& + \langle d(\varphi_1 w_l), d(\varphi_2(t)w_r) \rangle .
	\end{split}
	\label{eqn:Aprime}
\end{equation}
An upper bound similar to \eqref{eqn:boundright} and \eqref{eqn:boundleft} holds for $A_L'$ and $A_R'$. Together with \eqref{eqn:Wright}, we estimate
\[
\abs{\langle dA'_L,d(\varphi_2(t)w_r) \rangle }\leq C\lambda^4 e^{-2t} \chi_{[\log \lambda -\log \delta, \frac{1}{2}\log \lambda +1]}(t).
\]
Integrating over the cylinder, we obtain
\begin{eqnarray*}
	&& \int_{[\log \lambda-\log \delta,\log \delta]\times S^3} \abs{\langle dA'_L,d(\varphi_2(t)w_r) \rangle } d\omega dt\\
	&\leq& C\int_{[\log \lambda-\log \delta, \frac{1}{2}\log \lambda+1]}\lambda^4 e^{-2t} dt\\
	&\leq& \varepsilon(\delta) \lambda^2.
\end{eqnarray*}
The other term in the second line of \eqref{eqn:Aprime} is estimated similarly. For the third line, we use \eqref{eqn:Wright} and \eqref{eqn:Wleft} to see
\begin{eqnarray*}
	&& \int_{[\log \lambda-\log \delta,\log \delta]\times S^3} \abs{\langle d(\varphi_1(t)w_l), d(\varphi_2(t)w_r) \rangle } d\omega dt\\
	&\leq& C\int_{[\frac{1}{2}\log \lambda -1, \frac{1}{2}\log \lambda+1]}\lambda^4 e^{-t} dt\\
	&\leq& C(\delta) \lambda^{7/2}.
\end{eqnarray*}

Combining these upper bounds with \eqref{eqn:inter} and \eqref{eqn:Aprime}, we get
\begin{equation}
	\abs{E_{inter}- \int_{[\log\lambda -\log \delta,\log \delta]\times S^3} \langle dA'_R, dA'_L \rangle  d\omega dt}\leq \varepsilon(\delta) \lambda^2.
	\label{eqn:inter2}
\end{equation}
\begin{rem}\label{rem:later}
	Later, we shall fix $\delta$ first and consider the limit $\lambda\to 0$. In this sense, $C(\delta)\lambda^{7/2}$ is much smaller than $\varepsilon(\delta)\lambda^2$.
\end{rem}

To understand the interaction term in \eqref{eqn:inter2}, we remove the cut-off functions and define
\begin{equation}
	\begin{split}
		A''_L:=&  e^{2t}  \sum_{i=1,2,3} (F_{-,i}\theta_i' + F_{+,i} \psi_i') \\
		A''_R:=& \lambda^2 e^{-2t} \left( \psi_1' {\mathbf i} +\psi_2' {\mathbf j} + \psi_3' {\mathbf k} \right).
	\end{split}
	\label{eqn:Ano_cut}
\end{equation}

\begin{lem}\label{lem:vanish}
	For any $t\in [\log \lambda -\log \delta,\log \delta]$, we have
\[
\int_{\set{t} \times S^3} \langle dA''_R, dA''_L \rangle d\omega  =0.
\]
\end{lem}
\begin{proof}
	By \eqref{eqn:Ano_cut}, we have
	\begin{equation}\label{eqn:A2prime}
		\begin{split}
			dA''_L &= \sum_{i=1,2,3} F_{-,i} d(e^{2t}\theta_i') + F_{+,i} d(e^{2t}\psi'_i) \\
			dA''_R & = \lambda^2 \left(  d(e^{-2t} \psi_1') {\mathbf i} + d(e^{-2t}\psi_2') {\mathbf j} + d(e^{-2t} \psi_3'){\mathbf k} \right).
		\end{split}
	\end{equation}
	By Lemma \ref{lem:twoforms}, since SD forms are perpendicular to ASD forms, we get
	\begin{eqnarray*}
		\langle dA''_R,  dA''_L \rangle  &=& \lambda^2 \sum_{i=1,2,3} \langle d(e^{2t}\theta_i'), d(e^{-2t}\psi_1') \rangle (F_{-,i},{\mathbf i})_{\mathfrak g} \\
&& + \lambda^2 \sum_{i=1,2,3} \langle d(e^{2t}\theta_i'), d(e^{-2t}\psi_2') \rangle (F_{-,i},{\mathbf j})_{\mathfrak g} \\
&& + \lambda^2 \sum_{i=1,2,3} \langle d(e^{2t}\theta_i'), d(e^{-2t}\psi_3') \rangle (F_{-,i},{\mathbf k})_{\mathfrak g}. 
	\end{eqnarray*}
	Here we have used $(\cdot,\cdot)_{\mathfrak g}$ for the inner product of Lie algebra $\mathfrak g$. Notice that $F_{-,i}$'s are constant matrices, then this lemma is a corollary of Lemma \ref{lem:twist}.
\end{proof}

We continue to study the interaction term in \eqref{eqn:inter2},
\begin{equation}
	\begin{split}
	\langle dA'_R, dA'_L \rangle  =& \langle \varphi_4(t)dA''_R, \varphi_3(t) dA''_L \rangle  \\
	& + \langle \varphi_4'(t) dt\wedge A''_R, \varphi_3(t) dA''_L \rangle \\
	& + \langle \varphi_4(t) dA''_R, \varphi_3'(t) dt\wedge A''_L \rangle .
	\end{split}
	\label{eqn:3line}
\end{equation}
The integration of the first line over the cylinder vanishes due to Lemma \ref{lem:vanish}. For the second line, 
\begin{eqnarray*}
	dt\wedge A''_R &=& \lambda^2 e^{-2t}\left( dt\wedge \psi_1' {\mathbf i} + dt\wedge \psi_2'{\mathbf j} + dt\wedge \psi_3' {\mathbf k} \right).
\end{eqnarray*}
Using the obvious equality
\begin{equation}
4dt\wedge \psi_1' = e^{-2t} d(e^{2t}\psi_1')- e^{2t} d(e^{-2t}\psi_1'),
	\label{eqn:obv}
\end{equation}
we have
\begin{eqnarray*}
	&& \langle dt\wedge A''_R,  dA''_L \rangle \\
	&=& \frac{\lambda^2}{4} \langle \left( e^{-4t} d(e^{2t}\psi_1')- d(e^{-2t}\psi_1') \right){\mathbf i}, dA''_L \rangle + (\dots({\mathbf j},{\mathbf k})\dots).
\end{eqnarray*}
Notice that $dA''_L$ is a constant linear combination of $d(e^{2t}\theta_i')$ and $d(e^{2t}\psi_i')$, both of which are perpendicular to $d(e^{-2t}\psi_1')$ with respect to the $L^2$ inner product (of two forms) on $S^3$. To see this, we need the orthogonality between SD forms and ASD forms, and also Lemma \ref{lem:twist}.

Hence,
\begin{eqnarray*}
	&&\int_{\set{t}\times S^3} \langle dt\wedge A''_R, dA''_L \rangle \\
	&=& \int_{\set{t}\times S^3}\frac{\lambda^2}{4} \langle e^{-4t} d(e^{2t}\psi_1'){\mathbf i}, \sum_{i=1,2,3} F_{+,i}d(e^{2t}\psi'_i) \rangle + (\dots({\mathbf j},{\mathbf k})\dots) d\omega \\
	&=& \frac{\lambda^2}{4} e^{-4t} \norm{d(e^{2t}\psi_1')}_{L^2(S^3)}^2 ({\mathbf i},F_{+,1})_{\mathfrak g} + (\dots({\mathbf j},{\mathbf k})\dots) 
\end{eqnarray*}

Recall that $d(e^{2t}\psi_1')=2\Pi^*(dx_1\wedge dx_2 + dx_3\wedge dx_4)$. Since $\Pi$ is conformal and $\norm{\Pi^*(dx_i)}= e^t \norm{dx_i}$, we know 
\[
\frac{1}{4}e^{-4t} \norm{d(e^{2t}\psi_1')}_{L^2(S^3)}^2
\]
is a universal constant, which we denote by $c_0$.

In summary,
\begin{equation}
	\begin{split}
	& \int_{\set{t}\times S^3} \langle \varphi_4'(t) dt\wedge A''_R, \varphi_3(t) dA''_L \rangle d\omega \\
	=& \varphi_3(t)\varphi_4'(t) c_0 \lambda^2 \left( ({\mathbf i},F_{+,1})_{\mathfrak g} + ({\mathbf j},F_{+,2})_{\mathfrak g} + ({\mathbf k}, F_{+,3})_{\mathfrak g} \right).
	\end{split}
	\label{eqn:ijk}
\end{equation}
By definition, $\varphi_4'(t)$ is zero unless $t\in [\log \delta-2,\log \delta]$ and for such $t$, we always have $\varphi_3(t)=1$. Moreover, $\varphi'_4(t)$ is nonpositive and
\[
\int_{\log \delta-2}^{\log \delta} \varphi_4'(t)dt=-1.
\]
Hence, integrating \eqref{eqn:ijk} over $[\log \lambda -\log \delta, \log \delta]$ gives
\begin{equation}
 -c_0 \lambda^2 \left( ({\mathbf i},F_{+,1})_{\mathfrak g} + ({\mathbf j},F_{+,2})_{\mathfrak g} + ({\mathbf k}, F_{+,3})_{\mathfrak g} \right).
	\label{eqn:copy}
\end{equation}
The computation for the third line of \eqref{eqn:3line} is similar and the result is exactly the same as above. This is somewhat subtle, because at a first glance, we notice that $\varphi_3'$ is nonnegative instead of nonpositive (as is $\varphi'_4$). There is another minus sign in the computation that cancels this one.

For completeness, we list some key steps in this computation. By \eqref{eqn:A2prime}, $dA''_R$ is a linear combination of $d(e^{-2t}\psi_i')$. By \eqref{eqn:obv} and its analog for $dt\wedge d\theta'_i$, we know $dt\wedge A''_L$ is a linear combination of
\[
d(e^{2t}\theta_i'),\quad d(e^{-2t}\theta_i'),\quad d(e^{2t}\psi_i'),\quad d(e^{-2t}\psi_i').
\]
Because of the orthogonality in $L^2$ inner product of two forms as before, when we compute the integration of the pairing $\langle dt\wedge A''_L,  dA''_R \rangle $, it suffices to look at the coefficients of $d(e^{-2t}\psi_i')$,
\begin{eqnarray*}
	&& \int_{\set{t}\times S^3} \langle \varphi_3'(t) dt\wedge A''_L, \varphi_4(t) dA''_R \rangle d\omega \\
	&=& \varphi'_3(t)\varphi_4(t) \int_{\set{t}\times S^3}  \langle dt\wedge\sum_{i=1,2,3}(F_{+,i}e^{2t}\psi_i'), \lambda^2 d(e^{-2t}\psi_1'){\mathbf i} \rangle  d\omega + (\dots({\mathbf j},{\mathbf k})\dots) \\
	&=& \varphi'_3(t)\varphi_4(t)\frac{-1}{4} \int_{\set{t}\times S^3}  \langle \sum_{i=1,2,3}(F_{+,i}e^{4t}d(e^{-2t}\psi_i')), \lambda^2 d(e^{-2t}\psi_1'){\mathbf i} \rangle  d\omega + (\dots({\mathbf j},{\mathbf k})\dots)
\end{eqnarray*}
Here in the last line above, we used \eqref{eqn:obv} again and this time it is the second term that remains and introduces an extra minus sign in comparison with the previous computation. The rest of the computation is the same and gives us another copy of \eqref{eqn:copy}.

Now, Lemma \ref{lem:vanish}, the equations \eqref{eqn:3line} and \eqref{eqn:inter2} imply that
\[
\abs{E_{iner} + 2c_0 \lambda^2 \mathcal P} \leq \varepsilon(\delta)\lambda^2
\]
where
\begin{equation}
\mathcal P:= ({\mathbf i},F_{+,1})_{\mathfrak g}+ ({\mathbf j},F_{+,2})_{\mathfrak g} + ({\mathbf k},F_{+,3})_{\mathfrak g}
	\label{eqn:P}
\end{equation}
is a constant depending only on the SD part of the curvature (at the point ${\bf x}_0$) and a choice of basis in $\mathfrak s \mathfrak o(3)$.
Together with Lemma \ref{lem:loss} and \eqref{eqn:loss}, we obtain
\[
E_{gain}-E_{loss}= -2c_0\lambda^2 \mathcal P + \varepsilon(\delta) \lambda^2.
\]

Finally, if $\mathcal P$ is positive, then we may choose $\delta$ small so that
\[
-2c_0 \mathcal P + \varepsilon(\delta)<0.
\]
For sufficiently small $\lambda$(see Remark \ref{rem:later}), the new connection $D'$ satisfies \eqref{eqn:goal}.

Recall that in Section \ref{sec:pre}, when we wrote $D_{stan}$ into the form \eqref{eqn:oldA}, we have chosen ${\mathbf i}$, ${\mathbf j}$ and ${\mathbf k}$ to be a standard basis of $\mathfrak s \mathfrak u(2)=\mathfrak s \mathfrak o(3)$. Since the image of adjoint representation of $SO(3)$ is the set of all orientation-preserving orthogonal transformations of $\mathfrak s \mathfrak o(3)$, we may choose $({\mathbf i},{\mathbf j},{\mathbf k})$ to be any orthogonal basis with the given orientation. There exist good choices that make $\mathcal P>0$, because of the following trivial fact.

\begin{lem}
	\label{lem:trivial}
	Let $(F_1,F_2,F_3)$ be any three vectors in $\Real^3$ that are not all zero. There exists an orthonormal basis $(e_1,e_2,e_3)$ with any required orientation such that
	\begin{equation}
F_1\cdot e_1 + F_2\cdot e_2 + F_3\cdot e_3>0.
		\label{eqn:trivial}
	\end{equation}
\end{lem}
\begin{proof}
	Assume that $\abs{F_1}\geq \abs{F_2}\geq \abs{F_3}$. By $F_1\ne 0$, we take
\[
e_1= \frac{F_1}{ \abs{F_1}}.
\]

(1) In case that $F_2$ and $F_3$ are multiples of $F_1$, we can take any $e_2$ and $e_3$ to make a basis. The inequality \eqref{eqn:trivial} holds because $F_2\cdot e_2+ F_3\cdot e_3=0$.

(2) In case that $F_1$ and $F_2$ are linearly independent, set
\[
e_2= \frac{F_2- (F_2\cdot e_1)e_1}{ \abs{F_2- (F_2\cdot e_1)e_1}}
\]
and choose $e_3$ to be perpendicular to $e_1$ and $e_2$ with the right orientation. We then have $F_2\cdot e_2>0$ and 
\[
\abs{F_3\cdot e_3}\leq \abs{F_3}\leq \abs{F_1}= F_1\cdot e_1.
\]

(3) In case that $F_1$ and $F_3$ are linearly independent, the previous argument works by setting
\[
e_3= \frac{F_3- (F_3\cdot e_1)e_1}{ \abs{F_3- (F_3\cdot e_1)e_1}}
\]
and choosing $e_2$ according to the orientation. Noticing that $F_3\cdot e_3>0$ and
\begin{equation*}
	\abs{F_2\cdot e_2}\leq F_1\cdot e_1,
\end{equation*}
the proof is done.
\end{proof}

\section{Proof of the main theorem}\label{sec:proof}
In this section, we prove Theorem \ref{thm:main}.

(1) For any $p_1\in K_\eta$ with $-3\leq p_1\leq 3$, we take any minimizing sequence $D_i$ in $\mathcal C_{\eta,p_1}$ and run the Yang-Mills $\alpha$-flow as in Section \ref{subsec:alpha} to get a new minimizing sequence $D_i'$. Using the notations therein, we discuss the sign of $p_1-p_{1,\infty}$. 

If $p_1=p_{1,\infty}$, by Lemma \ref{lem:cor}, there is no bubble at all, which implies that $D_{\infty}$ is a minimizer in $\mathcal C_{\eta,p_1}$ and there is nothing to prove.

If $p_1<p_{1,\infty}$, Lemma \ref{lem:cor} and \eqref{eqn:ei} together imply that
\begin{equation}
\inf_{D'\in \mathcal C_{\eta,p_1}} \mathcal Y \mathcal M(D')= \mathcal Y \mathcal M(D_\infty) + 4\pi^2 (p_{1,\infty}-p_1).
	\label{eqn:over}
\end{equation}
Since $p_1-p_{1,\infty}$ is a multiple of $4$, $p_1<p_{1,\infty}$ and $-3\leq p_1\leq 3$ together imply that $p_{1,\infty}> 0$. It then follows from \eqref{eqn:pontryagin}, that there exists $\bx_0\in M$ satisfying
\[
\abs{(F_{D_\infty})^+}(\bx_0)>0.
\]
By Theorem 1.4, there is a connection $D'$ on $E_{\eta,p_1}$ such that
\[
\mathcal Y \mathcal M(D')< \mathcal Y \mathcal M(D_\infty) + 4\pi^2 (p_{1,\infty}-p_1).
\]
This is a contradiction to \eqref{eqn:over}.

The proof for the case $p_1>p_{1,\infty}$ is similar and omitted. The only difference is that we need to use an orientation reversing version of Theorem \ref{thm:tech}.

We give a proof of the part (3) in Theorem \ref{thm:main} only and the proof of the part (2) is the same.

For any $p_1\leq 0$ in $K_\eta$, we do a minimizing as before. Let $D_i'$ be the minimizing sequence in $\mathcal C_{\eta,p_1}$ given in Section \ref{subsec:alpha}. Again, we discuss the sign of $p_1-p_{1,\infty}$.

If $p_1=p_{1,\infty}$, then there is nothing to prove.

If $p_1> p_{1,\infty}$, we know $p_{1,\infty}<0$ and hence, for some $\bx_0\in M$, 
\[
(F_{D_\infty})^- (\bx_0)\ne 0.
\]
Now, we apply Theorem \ref{thm:tech} (with reversed orientation) to $D_\infty$ to get a connection $D'$ on $E_{\eta,p_1}$ with
\[
\mathcal Y \mathcal M(D')< \mathcal Y \mathcal M(D_\infty) + 4\pi^2 (p_1-p_{1,\infty}).
\]
However, by Lemma \ref{lem:cor} and \eqref{eqn:ei}, we know
\[
\inf_{D'\in \mathcal C_{\eta,p_1}} \mathcal Y \mathcal M(D')= \mathcal Y \mathcal M(D_\infty) + 4\pi^2 (p_1-p_{1,\infty}).
\]
This contradiction shows that $p_1>p_{1,\infty}$ is not possible.  A similar argument shows that $p_1\leq 0< p_{1,\infty}$ is not possible.

It remains to see what happens if $p_1<p_{1,\infty}\leq 0$. In this case, we claim that $D_\infty$ is ASD. Otherwise, we have $\bx_0\in M$ satisfying
\[
(F_{D_\infty})^+ (\bx_0) \ne 0,
\]
from which we draw a contradiction as before.

In summary, either for each nonpositive $p_1\in K_{\eta}$, the direct minimizing gives a minimizer in $\mathcal C_{\eta,p_1}$; or if one of them fail, we get an ASD connection $D_\infty$ in $\mathcal C_{\eta,p_{1,\infty}}$ with some $p_1<p_{1,\infty}\leq 0$.




\bibliographystyle{alpha}
\bibliography{foo}

\end{document}